\documentclass[11pt,twoside]{amsart}
\usepackage[all]{xy}   
\usepackage {amssymb,latexsym,amsthm,amsmath}
        
\long\def\symbolfootnote[#1]#2{\begingroup%
\def\thefootnote{\fnsymbol{footnote}}\footnote[#1]{#2}\endgroup}

\newcommand{\tr}{\ensuremath{{}^t\!}}
\newcommand{\tra}{\ensuremath{{}^t}}

\makeatletter
\def\imod#1{\allowbreak\mkern10mu({\operator@font mod}\,\,#1)}
\makeatother

\newtheorem{theorem}{Theorem}[section]
\newtheorem{lemma}[theorem]{Lemma}
\newtheorem{corollary}[theorem]{Corollary}
\newtheorem{proposition}[theorem]{Proposition}
\newtheorem{definition}[theorem]{Definition}
\newtheorem*{theorem*}{Theorem}
\theoremstyle{definition}

\newtheorem{example}[theorem]{Example}
\numberwithin{equation}{section}

\newcommand{\ignore}[1]{}

\newcommand{\mynote}[1]{}
\begin{document}
\setcounter{section}{0}
\title{Conjugacy classes of centralizers in unitary groups}
\author{Sushil Bhunia}
\author{Anupam Singh}
\address{IISER Pune, Dr. Homi Bhabha Road, Pashan, Pune 411008 INDIA}
\email{sushilbhunia@gmail.com}
\email{anupamk18@gmail.com}
\date{}
\thanks{The first named author gratefully acknowledges the PhD fellowship provided by CSIR during this work. The second named author acknowledges the support of NBHM grant during this work.}

\begin{abstract}
Let $G$ be a group. Two elements $x,y \in G$ are said to be in the same $z$-class if their centralizers in $G$ are conjugate within $G$. Consider $\mathbb F$ a perfect field of characteristic $\neq 2$, which has a non-trivial Galois automorphism of order $2$. Further, suppose that the fixed field $\mathbb F_0$ has the property that it has only finitely many field extensions of any finite degree. In this paper, we prove that the number of $z$-classes in the unitary group over such fields is finite. Further, we count the number of $z$-classes in the finite unitary group $U_n(q)$, and prove that this number is same as that of $GL_n(q)$ when $q>n$.
\end{abstract}
\maketitle
\section{Introduction}
Let $G$ be a group. Two elements $x$ and $y \in G$ are said to be $z$-equivalent, denoted as $x\sim_{z} y$, if their centralizers in $G$ are conjugate, i.e., $\mathcal Z_G(y)=g\mathcal Z_G(x)g^{-1}$ for some $g\in G$, where $\mathcal Z_G(x)$ denotes centralizer of $x$ in $G$. Clearly, $\sim_{z}$ is an equivalence relation on $G$. The equivalence classes with respect to this relation are called $z$-classes. It is easy to see that if two elements of a group $G$ are conjugate then their centralizers are conjugate thus they are also $z$-equivalent. However, in general, the converse is not true. We are interested in reductive linear algebraic groups, where a group may have infinitely many conjugacy classes but finitely many $z$-classes. In geometry, $z$-classes describe the behaviour of dynamical types. That is, if a group $G$ is acting on a manifold $M$ then understanding (dynamical types of) orbits is related to understanding (conjugacy classes of) centralizers. In this paper, we explore this topic for certain classical groups. Steinberg (see~\cite{St} Section 3.6 Corollary 1 to Theorem 2) proved that for a reductive algebraic group $G$ defined over an algebraically closed field, of good characteristic, the number of $z$-classes is finite (even though there could be infinitely many conjugacy classes). Thus, it is natural to ask how far ``finiteness of $z$-classes'' holds true for algebraic groups defined over a base field. This is certainly not true even for $GL_2$ over the field $\mathbb Q$ (see Section~\ref{example}). Thus we need to restrict to certain kind of fields.  
\begin{definition}[Property FE]\label{FE}
A perfect field $\mathbb{F}$ of characteristic $\neq 2$  has the property {\rm FE} if $\mathbb{F}$ has only finitely many field extensions of any finite degree. 
\end{definition}
\noindent Examples of such fields are an algebraically closed field (for example, $\mathbb C$), real numbers $\mathbb R$, local fields (e.g., $p$-adics $\mathbb Q_p$) and finite fields $\mathbb F_q$. 
In~\cite{Ku} for $GL_n$ and in~\cite{GK1} for orthogonal groups $O(V,B)$ and symplectic groups $Sp(V,B)$, it is proved that over a field with the property FE these groups have only finitely many $z$-classes. In this paper, we extend this result to the unitary groups. We prove the following result,    
\begin{theorem}\label{maintheorem}
Let $\mathbb F$ be a field with a non-trivial Galois automorphism of order $2$ and fixed field $\mathbb F_0$. Let $V$ be a finite dimensional vector space over $\mathbb F$ with a non-degenerate hermitian form $B$. Suppose that the fixed field $\mathbb{F}_0$ has the property {\rm FE}. Then, the number of $z$-classes in the unitary group $U(V,B)$ is finite.
\end{theorem}
\noindent  This theorem is proved in Section~\ref{prooftheorem1}. 

If we look at the character table of $SL_2(q)$ (for example see~\cite{B} or ~\cite{Pr}) we notice that the conjugacy classes and irreducible characters are grouped together. One observes a similar pattern in the work of Srinivasan~\cite{Sr} for $Sp_4(q)$. In~\cite{Gr}, Green studied the complex representations of $GL_n(q)$, where he introduced the function $t(n)$ for the `types of characters/classes' (towards the end of Section 1 on page 407-408) which is the number of $z$-classes in $GL_n(q)$. He observed we quote from Green (page 408 in~\cite{Gr}), the following:
\begin{quote}
The number $t(n)$ appears as the number of rows or columns of a character table of $GL_n(q)$. This is because the irreducible characters, which by a well-known theorem of representation theory are the same in number as the conjugacy classes, themselves collect into types in a corresponding way, and the values of all the characters of a given type at all the classes of a given type can be included in a single functional expression.
\end{quote}
\noindent In Deligne-Lusztig theory, one studies the representation theory of finite groups of Lie type and the $z$-classes of semisimple elements in the dual group play an important role. In~\cite{Hu}, Humphreys (Section 8.11) defined genus of an element in algebraic group $G$ over $\mathbb F$. Two elements have the same genus if they are $z$-equivalent in $G(\mathbb F)$ and the genus number (respectively semisimple genus number) is the number of $z$-classes (respectively the number of $z$-classes of semisimple elements). Thus understanding $z$-classes for finite groups of Lie type, especially semisimple genus, and their counting is of importance in representation theory (see~\cite{Ca, DM, Fl, FG}). Bose, in~\cite{Bo}, calculated genus number for simply connected simple algebraic groups over an algebraically closed field and compact simple Lie groups. Classification of $z$-classes in $U(n,1)$, the isometry group of the $n$-dimensional complex hyperbolic space is done in~\cite{CG}. The second named author studied $z$-classes for the compact group of type $G_2$ in \cite{Si}. In this paper, we count the $z$-classes in complex hyperbolic groups $U(n,1)$ and in finite unitary groups. The main theorem is as follows (for proof see Section~\ref{prooftheorem2}):
\begin{theorem}\label{maintheorem2}
The number of $z$-classes in $U_n(q)$ is same as the number of $z$-classes in $GL_n(q)$, when $q>n$. 
Thus, the number of $z$-classes for either group can be read off 
by looking at the coefficients of the function 
$\displaystyle\prod_{i=1}^{\infty} z(x^i)$, where $z(x)=\displaystyle\prod_{j=1}^{\infty}\frac{1}{(1-x^j)^{p(j)}}$.
\end{theorem}  
\noindent However, the above need not be true when $q\leq n$. We mention this in Example~\ref{example3}.

{\bf Acknowledgement: } The authors would like to thank Ben Martin, University of Aberdeen for his comments on the paper. The authors thank  referees for the wonderful comments which improved the readability of this paper.
 
\section{Conjugacy classes and centralizers in unitary groups}
In this section, we introduce the unitary groups. We begin with more general definition of hermitian forms, following~\cite{sc} Chapter 7, than what is available in~\cite{Gv} or in other textbooks on the subject of classical groups over a field. This is required for the description of conjugacy classes. Let $R$ be a commutative ring with a non-trivial involution $\sigma$ which, to simplify notation, we denote by $\bar a=\sigma(a)$. The subring of $R$ fixed by $\sigma$ is denoted by $R_0$. Let $V$ be a finitely generated projective module over $R$. A hermitian form on $V$ is a sesquilinear map $B\colon V\times V\rightarrow R$, satisfying $B(u,v)=\overline{ B(v,u)}$ for all $u,v \in V $. The pair $(V,B)$ is called a hermitian space. One can define when two hermitian spaces are equivalent (which is given in terms of an isometry defined in~\cite{sc} Chapter 7, Definition 2.1 (iv)). We list some results on classification of hermitian forms which will be required later in this paper. The following result follows from the main Theorem in~\cite{Ja}.
\begin{proposition}\label{Jacobson}
Let $\mathbb F$ be a field with a non-trivial Galois automorphism of order $2$. Suppose $\mathbb F_0$ has the property {\rm FE} and $V$ is a finite dimensional vector space over $\mathbb F$. Then there are only finitely many non-equivalent hermitian forms on $V$. 
\end{proposition}
\begin{proof}
From Jacobson's theorem (see Theorem~\cite{Ja}), the classification of hermitian forms over $\mathbb F$ is same as the classification of symmetric bilinear forms over $\mathbb F_0$. Since symmetric bilinear forms can be diagonalised where the non-zero entries belong to $\mathbb F_0^{\times}/{\mathbb F_0^{\times}}^2$, which is given to be finite (as $\mathbb F_0^{\times}/{\mathbb F_0^{\times}}^2$ determines the distinct degree $2$ field extensions of $\mathbb F_0$ which is finite as $\mathbb F_0$ has the property FE), this gives the required result.
\end{proof}
We need to understand Hermitian forms on a module $V$ over ring $A=\mathbb F\times \mathbb F$, where $\mathbb F$ is a field. We consider $A$ as an algebra over $\mathbb F$ with diagonal embedding and the ``switch'' involution on $A$ given by $(a,b)\mapsto (b,a)$. Then we have orthogonal idempotents $e_1=(1,0), e_2=(0,1)$ in $A$ satisfying $e_1+e_2=1$ and $e_1e_2=0$. We decompose $V=V_1\oplus V_2$, where $V_1=e_1V$ and $V_2=e_2V$. Clearly, $V_1$ and $V_2$ are vector spaces over $\mathbb F$, say, of dimension $n$ and $m$ respectively. Then, 
\begin{proposition}\label{ffhermitian}
With notation as above, any hermitian form on $A$-module $V$, up to equivalence, is determined by the Smith normal form (equivalently the rank) of an $n\times m$ matrix over $\mathbb F$.
\end{proposition}
\begin{proof}
For convenience, we fix a basis of $V_1$ as $\{\epsilon_1,\ldots,\epsilon_n\}$ and of $V_2$ as $\{\delta_1,\ldots,\delta_m\}$ which together give a generating set for $A$-module $V$. Let $B$ be a sesquilinear form on $V$. Then, using orthogonality of idempotents we get, $B(v_i,v_i)=0$ for $v_i\in V_i$. Thus, $B$ is determined by the values $B(\epsilon_i, \delta_j)$ and $B(\delta_j, \epsilon_i)$ for all $i$ and $j$. Further, if $B$ is a hermitian form then $B$ is determined by the values $B(\epsilon_i, \delta_j)$, a $n\times m$ matrix over $A$. That is, $B$ is determined by its restriction on $V_1\times V_2$, say by a matrix $\beta\in M_{n\times m}(A)$. Hence, the matrix of $B$ would look like $\begin{pmatrix} 0 & \beta \\\tr {\bar \beta} & 0\end{pmatrix}$. We can also check that any isomorphism $P$ of $V$ will be given by $\begin{pmatrix} P_1 &0 \\ 0 & P_2\end{pmatrix}$, where $P_1\in GL(V_1)$ and $P_2\in GL(V_2)$. To get the equivalence of forms we compute $\tr P B\bar P$ and get $\tr P_1\beta \bar P_2$. Thus, $\beta$ is determined by its Smith-normal form. The converse can be easily checked by an explicit construction.
\end{proof}

Usually, to define the unitary group we begin with a field $\mathbb F$ and a non-trivial Galois automorphism $\sigma$ of order $2$ on it. We consider only non-degenerate hermitian forms. An isometry of $(V,B)$ is a linear map $g\in GL(V)$ such that $B(u,v)=B(g(u),g(v))$ for all $u,v\in V$. Such maps are called unitary transformations and the set of all such transformations is called  the {\bf unitary group} $U(V,B)$. There could be more than one hermitian form, up to equivalence, over a given field and similarly more than one non-isomorphic unitary group. For example, when $\mathbb F=\mathbb C$ with conjugation then the hermitian forms are given by signature. However, over a finite field, there is a unique hermitian form (see~\cite{Gv} Corollary 10.4). We discuss some of the cases in Section~\ref{counting}.

To study $z$-classes, it is important to understand the conjugacy classes first. This has been well understood for classical groups through the work of Asai, Ennola, Macdonald, Milnor, Springer-Steinberg, Wall, Williamson~\cite{As,En,Ma,Mi,SS,Wa,Wi} and many others. Since the unitary group is a subgroup of $GL_n(\mathbb F)$ one hopes to exploit the theory of canonical forms to get the conjugacy classes in the unitary group. For our exposition, we follow Springer-Steinberg~\cite{SS}. We begin with recalling notation involved in the description of conjugacy classes.

\subsection{Self-U-reciprocal polynomials}
Let $f(x)=\sum_{i=0}^d a_{i}x^i$ be a polynomial in $\mathbb F[x]$ of degree $d$. We extend the involution on $\mathbb F$ to that of $\mathbb F[x]$ by $\overline{f}(x):=\sum_{i=0}^d \bar{a_{i}} x^i$. Let $f(x)$ be a polynomial with $f(0)\neq 0$. The corresponding U-reciprocal polynomial of $f(x)$ is a degree $d$ polynomial defined by 
$$\tilde{f}(x):=\overline{f}(0)^{-1}\ x^d \  \overline{f}(x^{-1}).$$
A monic polynomial $f(x)$ with a non-zero constant term is said to be \textbf{self-U-reciprocal} if $f(x)=\tilde{f}(x)$. 
In terms of roots, it means that for a self-U-reciprocal polynomial, whenever $\lambda$ is a root,  $\bar{\lambda}^{-1}$ is also a root with the same multiplicity. Note that $f(x)=\tilde{\tilde{f}}(x)$, and if $f(x)=f_{1}(x)f_{2}(x)$ 
then $\tilde{f}(x)=\tilde{f_{1}}(x)\tilde{f_{2}}(x)$. Also, $f(x)$ is irreducible if and only if $\tilde{f}(x)$ is 
irreducible. In the case of $f(x)=(x-\lambda)^n$, the polynomial $f(x)$ is self-U-reciprocal if and only if 
$\lambda \bar{\lambda}=1$. Over a finite field, we have the following result due to Ennola (see~\cite{En} Lemma 2):
\begin{proposition}\label{irr-self-U-rec}
Let $f(x)$ be a monic, irreducible, self-U-reciprocal polynomial over a finite field $\mathbb F_{q^2}$. 
Then the degree of $f(x)$ is odd.    
\end{proposition}
Let $T\in GL(V)$ and $f(x)$ be its minimal polynomial then $\tilde{f}(x)$ is the minimal polynomial of $\bar{T}^{-1}$. If $T \in U(V,B)$ then its minimal polynomial is monic with a non-zero constant term and is self-U-reciprocal. Let $f(x)$ be a self-U-reciprocal polynomial. We can write it as follows,
\begin{equation}\label{u-rec} 
f(x)=\prod_{i=1}^{k_1} p_{i}(x)^{m_{i}}\prod_{j=1}^{k_{2}}(q_{j}(x)\tilde{q_{j}}(x))^{n_{j}},
\end{equation}
where $p_{i}(x)$ is irreducible, self-U-reciprocal, and $q_{j}(x)$ is irreducible, not self-U-reciprocal for all $i$ and $j$.  
 
\subsection{Space decomposition with respect to a unitary transformation}
Let $T \in U(V, B)$ and $f(x)$ be its minimal polynomial. Write $f(x)$ as in the Equation~\ref{u-rec}. This gives primary decomposition of $V$. Further we have the following,
\begin{proposition}\label{primary-decom}
The direct sum decomposition $V=\bigoplus_{i}\mathrm{ker}(f_{i}(T)^{s_{i}})$, where either $f_i(x) = p_i(x)$ and $s_i=m_i$ or $f_i(x)=q_i(x)\tilde q_i(x)$ and $s_i=n_i$, is a decomposition into non-degenerate mutually orthogonal $T$-invariant subspaces.
\end{proposition}
\noindent The proof of this is similar to the orthogonal case as in the Section 3 of~\cite{GK1} and hence we skip the details. This decomposition helps us reduce the questions about conjugacy classes and $z$-classes of a unitary transformation to the unitary transformations with the minimal polynomial of one of the following two kinds:
\begin{itemize}\label{poly-type}
\item[Type 1.] $p(x)^{m}$, where $p(x)$ is monic, irreducible, self-U-reciprocal polynomial with a  non-zero constant term, 
\item[Type 2.] $(q(x)\tilde{q}(x))^{m}$, where $q(x)$ is monic, irreducible, not self-U-reciprocal polynomial with a non-zero constant term.
\end{itemize}
The Proposition~\ref{primary-decom} above gives us a primary decomposition of $V$ into $T$-invariant $B$ non-degenerate subspaces
\begin{equation}\label{primary}
V=\left(\bigoplus_{i=1}^{k_1}V_{i}\right)\bigoplus \left( \bigoplus_{j=1}^{k_2}V_{j} \right),
\end{equation}
where $V_{i}$ corresponds to the polynomials of Type 1 and $V_{j}=V_{q_{j}}+ V_{\tilde{q_{j}}}$ corresponds to the polynomials of Type 2. Denote the restriction of $T$ to each $V_r$ by $T_{r}$. Then the minimal polynomial of $T_{r}$ is one of the two types. It turns out that the centralizer of $T$ in $U(V,B)$ is 
$$\mathcal Z_{U(V,B)}(T)=\prod_{r} \mathcal Z_{U(V_r, B_r)}(T_{r}),$$
where $B_r$ is a hermitian form obtained by restricting $B$ to $V_r$. The direct product here comes from the primary decomposition (see~\cite{SS} Chapter IV, 2.8). Thus the conjugacy class and the $z$-class of $T$ is determined by the restriction of $T$ to each of the primary subspaces. Hence it is enough to determine the conjugacy class and the $z$-class of $T\in U(V,B)$, which has minimal polynomial of one of the types listed above. 
  
\subsection{Conjugacy classes and centralizers}
Now, let us define $E^T=\frac{\mathbb{F}[x]}{<f(x)>}$, an $\mathbb F$ algebra. Then $V$ is an $E^T$-module, where $x$ acts via $T$. To keep track of the action we denote this module by $V^T$ although it's simply $V$ as an $\mathbb F$-vector space. The $E^T$-module structure on $V^T$ determines $GL_n$-conjugacy class of $T$. To determine conjugacy class of $T$ within $U(V,B)$, Springer and Steinberg (see~\cite{SS} Chapter IV, 2.6) defined a non-degenerate hermitian form $H^T$ on $V^T$ induced from $B$ and $T$. We briefly recall this here. Since $f(x)$ is self-U-reciprocal polynomial, there exists a unique involution $\alpha$ on $E^T$ such that $\alpha(x)=x^{-1}$ and $\alpha$ is an extension of $\sigma$ on scalars. Thus, $(E^T, \alpha)$ is an algebra with involution. Further, there exists a $\mathbb{F}$-linear function $l^T\colon E^T \rightarrow \mathbb F$ such that the symmetric bilinear form $\bar{l^T}\colon E^T \times E^T \rightarrow \mathbb F$ given by $\bar{l^T}(a,b)=l^T(ab)$ is non-degenerate with $l^T(\alpha(a))=l^T(a)$ for all $a\in E^T$. The hermitian form $H^T$ on $E^T$-module $V^T$ (with respect to $\alpha$) satisfies $B(eu, v) = l^T(eH^T(u,v))$ for all $e\in E^T$ and $u,v \in V^T$ and is non-degenerate. 
Then we have (see~\cite{SS} Chapter IV, 2.7 and 2.8), 
\begin{proposition}\label{conjugacy}
With notation as above, let $S$ and $T \in U(V,B)$. Then,
\begin{enumerate}
\item the elements $S$ and $T$ are conjugate in $U(V,B)$ if and only if 
\begin{enumerate}
\item there exists an algebra-isomorphism $\psi \colon E^S \rightarrow E^T$ induced by $x$ going to $x$, and
\item an $\mathbb F$-isomorphism $\phi\colon V^S \rightarrow V^T$ satisfying $\phi(a v)=\psi(a)\phi(v)$ and $H^T(\phi(v),\phi(w))=\psi\left(H^S(v,w)\right)$ for all $a \in E^S$ and $v,w\in V^S$.
\end{enumerate}
\item The centralizer $\mathcal Z_{U(V,B)}(T)=U(V^T, H^T)$.
\end{enumerate}
\end{proposition}
We can decompose the algebra $E^T$ as a direct sum of subalgebras with respect to $\alpha$ as $E^T=E_{1}\oplus E_{2} \oplus\cdots \oplus E_{r}$, where each $E_{i}$ is $\alpha$-indecomposable (see~\cite{SS} Chapter IV, Section 2.2). That is, $E_i$ can not be further written as a direct sum of $\alpha$-stable subalgebras. Now denote the restriction of $\alpha$ to $E_i$ by $\alpha_{i}$. The $\alpha_i$, thus obtained, is again an involution on $E_i$. Clearly, $E_{i}$'s are of one of the following types (recall the decomposition of $f(x)$ in Equation~\ref{u-rec}, also see~\cite{As} Lemma 2.6). 
\begin{itemize}
\item $\frac{\mathbb{F}[x]}{<p(x)^d>}$, where $p(x)$ is a monic, irreducible, self-U-reciprocal polynomial.
\item $\frac{\mathbb{F}[x]}{<q(x)^d>} \oplus \frac{\mathbb{F}[x]}{<\tilde{q}(x)^d>}$, where $q(x)$ is a monic, irreducible but not self-U-reciprocal.
\end{itemize}
In the second case, the two components $\frac{\mathbb{F}[x]}{<q(x)^d>)}$ and $\frac{\mathbb{F}[x]}{<\tilde{q}(x)^d>}$ are isomorphic local rings (induced by the U-reciprocal polynomial structure) and the restriction of $\alpha$ is given by $\alpha(a,b)=(b,a)$ via the fixed isomorphism. Using Wall's approximation theorem (see Corollary~\ref{Wall's} in the next section) it is easy to see that all non-degenerate hermitian forms over such rings are equivalent. Thus to determine equivalence of $H^T$ we need to look at modules over rings of the first Type above. 

\subsection{Wall's approximation theorem}
We have reduced the conjugacy problem to equivalence of hermitian forms over certain rings. However, it turns out (see the corollary below) that the second case is easy. Let $R$ be a commutative ring, where $2$ is invertible (so that it satisfies the trace condition) and $J$ be its Jacobson radical, and $\alpha$ be an involution on $R$. Let $M$ be a finitely generated module over $R$ and $(M,B)$ be a non-degenerate hermitian space. We define $\underline {M}:= \frac{M}{JM}$ a module over $\underline {R}:=\frac{R}{J}$. 
Now $B$ induces a hermitian form $\underline {B}$ on $\underline{M}$ with respect to 
the involution $\underline{\alpha}$ of $\underline{R}$ induced by $\alpha$. Then it is easy to see that $(\underline{M},\underline{B})$ is non-degenerate. The converse is a theorem of Wall (see~\cite{Wa} Theorem 2.2.1, and also~\cite{As} Proposition 2.5) which would be useful for further analysis. For the general theory of hermitian forms over a ring (which need not be commutative) we refer the reader to~\cite{sc} Chapter 7, Theorem 4.4.
\begin{theorem}[Wall's approximation theorem]
With notation as above,
\begin{enumerate}
\item any non-degenerate hermitian form over $\underline{R}$ is induced by some non-degenerate hermitian form over $R$.
\item Let $(M_{1}, B_{1})$ and $ (M_{2}, B_{2})$ be non-degenerate hermitian spaces over $R$ and correspondingly  $(\underline{M_{1}}, \underline{B_{1}})$ and $(\underline{M_{2}}, \underline{B_{2}})$ be non-degenerate hermitian spaces over $\underline{R}$. Then $(M_{1}, B_{1})$ is equivalent to $(M_{2}, B_{2})$ if and only if $(\underline{M_{1}}, \underline{B_{1}})$ is equivalent to $(\underline{M_{2}}, \underline{B_{2}})$.
\end{enumerate}
\end{theorem}
\noindent We need the following,
\begin{corollary}\label{Wall's}
Let $\mathbb F$ be a perfect field of characteristic $\neq 2$. Let $V$ be a finitely generated module over $A=\frac{\mathbb{F}[x]}{<q(x)^d>} \oplus \frac{\mathbb{F}[x]}{<\tilde{q}(x)^d>}$, where $q(x)$ is a monic, irreducible polynomial and  $H_{1}$ and $H_{2}$ be two non-degenerate hermitian forms on $V$ with respect to the ``switch'' 
involution on $A$ given by $\overline{(b,a)} = (a,b)$. Then $H_{1}$ and $H_{2}$ are equivalent.
\end{corollary}
\begin{proof}
We use Wall's approximation theorem. Here $R=A$ and its Jacobson radical is 
$J=\frac{<q(x)>}{<q(x)^d>} \oplus \frac{<\tilde q(x)>}{<\tilde q(x)^d>}$. 
Then $\underline {R} =\underline{A} \cong \frac{\mathbb{F}[x]}{<q(x)>} \oplus \frac{\mathbb{F}[x]}{<\tilde q(x)>}\cong K \oplus K$ is a semisimple ring, where $K\cong \frac{\mathbb{F}[x]}{<q(x)>} \cong \frac{\mathbb{F}[x]}{<\tilde q(x)>}$ is a finite field extension of $\mathbb{F}$ (thus separable). Now we have hermitian forms 
$\underline{H_{i}}\colon  \underline{V}\times \underline{V}\rightarrow \underline{A}$ defined by 
$\underline{H_{i}}(u+JV, v+JV)= H_{i}(u,v)J$ for all $u,v \in V$. Now, the result follows from Proposition~\ref{ffhermitian} and noting that we have $n=m$ and forms are non-degenerate.
\end{proof}
   
\subsection{Unipotent elements}
We look at the Type 1 more closely, where the minimal polynomial is $p(x)^d$ with $p(x)$ an irreducible, self-U-reciprocal polynomial. This includes unipotent elements. The theory of rational canonical forms, which determines conjugacy classes in $GL(V)$, gives a decomposition of $V= \oplus_{i=1}^kV_{d_{i}}$ with $1\leq d_{1}\leq d_{2} \leq \ldots \leq d_{k}=d$ and each $V_{d_{i}}$ is a free module over $\mathbb{F}$-algebra $\frac{\mathbb{F}[x]}{<p(x)^{d_{i}}>}$ (see~\cite{SS} Chapter IV, 2.14). Thus,
\begin{proposition}\label{p-conj}
Let $S$ and $T$ be in $U(V,B)$. Suppose the minimal polynomial of both $S$ and $T$ are equal, and it equals $p(x)^d$, where $p(x)$ is an irreducible self-$U$-reciprocal polynomial. Then $S$ and $T$ are conjugate in $U(V,B)$ if and only if 
\begin{enumerate}
\item the elementary divisors $p(x)^{d_i}$, where $1\leq d_{1}\leq d_{2} \leq \ldots \leq d_{k} =d$ of $S$ and $T$ are same, i.e., they are $GL(V)$ conjugate, and
\item the sequence of non-degenerate hermitian spaces, $\left\{(V_{d_{1}}^S, H_{d_{1}}^S), \ldots,(V_{d_{k}}^S,H_{d_{k}}^S)\right\}$ corresponding to $S$ and  $\left\{(V_{d_{1}}^T, H_{d_{1}}^T),\ldots,(V_{d_{k}}^T,H_{d_{k}}^T)\right\}$ corresponding to $T$ are equivalent. Note that, $H_{d_{i}}^S$ and $H_{d_{i}}^T$ take values in the cyclic $\mathbb{F}$-algebra $\frac{\mathbb{F}[x]}{<p(x)^{d_{i}}>}$.
\end{enumerate}
Further, the centralizer of $T$, in this case, is the direct product $\mathcal Z_{U(V,B)}(T)=\displaystyle{\prod_{i=1}^k} U(V_{d_{i}}^T, H_{d_{i}}^T)$.
\end{proposition}
\noindent This gives us the following,
\begin{corollary}\label{zunipotent}
Let $\mathbb F_0$ has the property {\rm FE}. Then, the number of conjugacy classes of unipotent elements in $U(V, B)$ is finite. And hence, the number of $z$-classes of unipotent elements in $U(V, B)$ is finite.
\end{corollary}
\begin{proof}
We use the above Proposition~\ref{p-conj}. A unipotent element has minimal polynomial $(x-1)^d$ for some $d$, that is, we have $p(x)= x-1$. Then, the conjugacy classes of unipotents correspond to a sequence 
$1\leq d_{1}\leq d_{2} \leq \ldots \leq d_{k} =d$, and non-degenerate hermitian spaces 
$\{(V_{d_1}^T, H_{d_1}^T), \ldots, (V_{d_{k}}^T, H_{d_{k}}^T)\}$ up to equivalence. 
Now $\underline {E}_{d_{i}}^T=\frac{\mathbb{F}[T]}{<T-1>} \cong \mathbb{F}$. 
Then, by the Wall's approximation theorem, the number of non-equivalent non-degenerate hermitian forms $(V, B)$ is exactly equal to the number of non-equivalent hermitian forms $(\underline{V}, \underline{B})$. From Proposition~\ref{Jacobson}, we know that there are only finitely many non-equivalent hermitian forms over $\mathbb{F}$. 
Thus $H_{d_i}^T$ has only finitely many choices for each $i$. Hence the result.
\end{proof}
   
\section{$z$-classes in Unitary Groups and fields with property FE}
A unitary group is an algebraic group defined over $\mathbb F_0$. 
Since we are working with perfect fields, an element $T\in U(V,B)$ 
has a Jordan decomposition, $T=T_{s}T_{u}=T_uT_s$, where $T_{s}$ is semisimple and $T_{u}$ is unipotent. 
Further one can use this to compute the centralizer $\mathcal Z_{U(V,B)}(T)=
\mathcal Z_{U(V,B)}(T_{s})\cap \mathcal Z_{U(V,B)}(T_{u})$. So, the Jordan decomposition 
helps us reduce the study of conjugacy and computation of the centralizer of an element to the study 
of that of its semisimple and unipotent parts. In this section, we analyze semisimple elements and then we prove our main theorem.

\subsection{Semisimple $z$-classes}
Let $T\in U(V,B)$ be a semisimple element. First, we begin with a basic case.
\begin{lemma}\label{zself}
Let $T\in U(V,B)$ be a semisimple element such that the minimal polynomial is either $p(x)$, which is irreducible, self-U-reciprocal or $q(x)\tilde q(x)$, where $q(x)$ is irreducible not self-U-reciprocal. Let $E=\frac{\mathbb{F}[x]}{<p(x)>}$ in the first case and $\frac{\mathbb{F}[x]}{<q(x)>}$ in the second case. Then, the $z$-class of $T$ is determined by the following:
\begin{enumerate}
\item the algebra $E$ over $\mathbb F$, and
\item the equivalence class of $E$-valued hermitian form $H^T$ on $V^T$.
\end{enumerate}
\end{lemma}
\begin{proof} 
This is the case of regular semisimple element (see~\cite{FG} Section 3). Suppose $S, T \in U(V,B)$ are in the same $z$-class, then 
$\mathcal Z_{U(V,B)}(S)=g\mathcal Z_{U(V,B)}(T)g^{-1}$ for some $g\in U(V,B)$. 
We may replace $T$ by its conjugate $gTg^{-1}$, so we get 
$\mathcal Z_{U(V,B)}(S)=\mathcal Z_{U(V,B)}(T)$ which is a maximal torus.  
Since $T$ is semisimple and regular (the condition on the minimal polynomial gives this) $\mathcal Z_{End(V)}(T)\cong E$ in the first case and $E\times E$ in the second case, which consists of essentially polynomials in $T$.
Thus, $U(V^T, H^T)=\mathcal Z_{End(V)}(T) \cap U(V,H^T)$ and hence $T$ uniquely determines $E$ and $H^T$. The converse follows from Proposition~\ref{conjugacy}.
\end{proof}
Now for the general case, let $T\in U(V,B)$ be a semisimple element with minimal polynomial written as in Equation~\ref{u-rec}
$$m_{T}(x)=\prod_{i=1}^{k_1} p_{i}(x) \prod_{j=1}^{k_{2}}\left(q_{j}(x)\tilde{q}_{j}(x)\right),$$ 
where $p_{i}(x)$ are irreducible, self-U-reciprocal polynomials of degree $m_i$ and $q_{j}(x)$ are irreducible but not self-U-reciprocal of degree $l_j$. 
Let the characteristic polynomial of $T$ be $$\chi_{T}(x)=\prod_{i=1}^{k_1} p_{i}(x)^{d_i} \prod_{j=1}^{k_{2}}\left(q_{j}(x)\tilde{q}_{j}(x)\right)^{r_j}.$$
Let us write the primary decomposition of $V$ with respect to $m_T$ into $T$-invariant subspaces as 
\begin{equation}
V=\bigoplus_{i=1}^{k_1} V_{i}\bigoplus\bigoplus_{j=1}^{k_{2}}\left(W_{j}+ \tilde W_{j}\right).
\end{equation}
Denote $E_{i}=\frac{\mathbb{F}[x]}{<p_{i}(x)>}$ and $K_{j}=\frac{\mathbb{F}[x]}{<q_{j}(x)>}$ the field extensions of $\mathbb{F}$ of degree $m_{i}$ and $l_{j}$ respectively.
\begin{theorem}\label{zsemisimple}
With notation as above, let $T\in U(V,B)$ be a semisimple element. Then the $z$-class of $T$ is determined by the following:
\begin{enumerate}
\item A finite sequence of integers $(m_{1},\ldots, m_{k_{1}}; l_{1},\ldots, l_{k_{2}})$ each $\geq 1$ such that 
$$n = \displaystyle \sum_{i=1}^{k_1} d_{i}m_{i} + 2 \sum_{j=1}^{k_{2}}r_{j}l_{j}.$$
\item Finite field extensions $E_{i}$ of $\mathbb{F}$ of degree $m_{i}$ for $1\leq i \leq k_{1}$, and $K_{j}$ of $\mathbb{F}$ of degree $l_{j}$, for $1\leq j \leq k_{2}$.
\item Equivalence classes of $E_{i}$-valued hermitian forms $H_{i}$ of rank $d_i$ and $K_j\times K_j$-valued hermitian forms $H_{j}'$ (which is unique up to equivalence) of rank $r_j$.
\end{enumerate}
Further, with this notation, $\mathcal Z_{U(V,B)}(T)\cong \displaystyle \prod_{i=1}^{k_1} U_{d_i}(H_i) \times \prod_{j=1}^{k_2} GL_{r_j}(K_j)$.
\end{theorem}
\begin{proof}
The proof of this follows from Proposition~\ref{conjugacy} and Lemma~\ref{zself}.
\end{proof}
\begin{corollary}\label{zsemisimpleFE}
Let $\mathbb F_0$ have the property {\rm FE}. Then the number of semisimple $z$-classes in $U(V,B)$ is finite.
\end{corollary}
\begin{proof}
This follows if we show that there are only finitely many hermitian forms over $\mathbb F$ up to equivalence of any degree $n$ which is the content of Proposition~\ref{Jacobson}. 
\end{proof}

\subsection{Proof of the Theorem~\ref{maintheorem}}\label{prooftheorem1}
It is already known that the number of $z$-classes in $GL_n(\mathbb F)$ is finite when $\mathbb F$ has the property FE (see~\cite{Ku} page 323 (iv)).
The number of conjugacy classes of centralizers of semisimple elements is finite follows from 
Corollary~\ref{zsemisimpleFE}. Hence, up to conjugacy, there are finitely many possibilities 
for $\mathcal Z_{U(V,B)}(s)$ for $s$ semisimple in $U(V,B)$. Let $T\in U(V,B)$, then it has a 
Jordan decomposition $T=T_{s}T_{u}=T_{u}T_{s}$. Recall $\mathcal Z_{U(V,B)}(T)
=\mathcal Z_{U(V,B)}(T_{s})\cap \mathcal Z_{U(V,B)}(T_{u})$ and $T_{u} \in Z_{U(V,B)}(T_{s})^{\circ}$ (see~\cite{SS} Chapter V, 3.16). 
Now, $\mathcal Z_{U(V,B)}(T_{s})$ is a product of certain unitary groups and general linear groups 
possibly over a finite extension of $\mathbb F$ (e.g. $E_i$ and $K_j$ in Theorem~\ref{zsemisimple}). Since $\mathbb F_0$ has the property FE, any finite extension of it has this property. Thus we can apply Corollary~\ref{zunipotent} to the 
group $\mathcal Z_{U(V,B)}(T_{s})$ and get, up to conjugacy, $T_{u}$ has finitely many 
possibilities in $\mathcal Z_{U(V,B)}(T_{s})$. Hence, up to conjugacy, $\mathcal Z_{U(V,B)}(T)$ 
has finitely many possibilities in $U(V,B)$. Therefore the number of $z$-classes in $U(V,B)$ is finite.

\section{Counting $z$-classes in Unitary Group}\label{counting}

We recall that there could be more than one non-equivalent non-degenerate hermitian form over 
a given field $\mathbb F$ and hence more than one non-isomorphic unitary group. In this section, we want to count the number of $z$-classes and write its generating function.
Special focus is on the unitary group over finite field $\mathbb F=\mathbb F_{q^2}$ of characteristic $\neq 2$ with $\sigma$ given by $\bar x=x^q$ and $\mathbb F_0=\mathbb F_q$. It is well-known that over a finite field there is a unique non-degenerate hermitian form, up to equivalence, thus a unique unitary group up to conjugation. We denote the unitary group by $U_n(q)=\{g\in GL_n(q^2)\mid \tra gJ\bar g=J\}$, where $J$ is an invertible hermitian matrix (for example, the identity matrix). The groups $GL_n(q)$ and $U_n(q)$ both are subgroups of $GL_n(q^2)$. It is interesting to note that the sizes of these groups as a function of $q$ can be obtained from one another by $q\leftrightarrow -q$. Ennola noted that this idea goes forward to the sizes of conjugacy classes and representations of small rank unitary groups. This came to be known as Ennola duality and later it was proved that indeed the Green's polynomial for unitary groups could be obtained this way from that of $GL_n(q)$. Thus, the representation theory of both these groups is closely related (for example see~\cite{SV, TV}). For applications in the subject of derangements see Burness and Giudici~\cite{BG}. Thus it is always useful to compare any computation for $U_n(q)$ with that of $GL_n(q)$. We begin with recording some well-known results about $GL_n$. 

\subsection{$z$-classes in general linear group}
Let $p(n)$ denote the number of partitions of $n$ with generating function 
$p(x):=\displaystyle\sum_{n=0}^{\infty} p(n)x^n = \displaystyle\prod_{i=1}^{\infty} \frac{1}{1-x^i}$. 
Let $z_{\mathbb F}(n)$ denote the number of $z$-classes in $GL_n(\mathbb F)$ and the generating function be $z_{\mathbb F}(x) := \displaystyle \sum_{n=0}^{\infty} z_{\mathbb F}(n)x^n$. 
Let us denote a partition of $n$ by $(1^{k_1} 2^{k_2}\ldots n^{k_n})$ with $n=\sum_{i} ik_i$. 
For convenience, we denote this by $n\vdash (1^{k_1} 2^{k_2}\ldots n^{k_n})$. The following result is  a consequence of the theory of Jordan canonical forms and the formula is given in~\cite{Ku} (Section 10, The absolute case). However the formula there has a printing error.
\begin{proposition}
Let $K$ be an algebraically closed field. Then, 
\begin{enumerate}
\item the number of $z$-classes of semisimple elements in $GL_n(K)$ is $p(n)$, and is equal to the number of $z$-classes of unipotent elements. 
\item The number of $z$-classes in $GL_{n}(K)$ is 
$$z_K(n) = \displaystyle \sum_{n\vdash (1^{k_1} 2^{k_2}\ldots n^{k_n})} \prod_{i=1}^{n} \binom{p(i)+k_i-1}{k_i},
$$
and the generating function is 
$$z(x):=z_K(x)=\prod_{i=1}^{\infty}\frac{1}{(1-x^i)^{p(i)}}.$$
\end{enumerate}
\end{proposition}
\noindent  Green computed the number of $z$-classes in $GL_n(q)$ (see~\cite{Gr} Section 1), which is the function $t(n)$ there. We list them here in our notation (see~\cite{Ku} Section 10, General case).
\begin{proposition}
Let $z(x)=\displaystyle\prod_{i=1}^{\infty}\frac{1}{(1-x^i)^{p(i)}}$. Then,
\begin{enumerate}
\item $z_{\mathbb C}(x)=z(x)$.
\item $z_{\mathbb R}(x)= z(x)z(x^2)$.
\item When $q>n$ the function $\displaystyle\prod_{i=1}^{\infty} z(x^i)$ computes $z_{\mathbb F_q}(n)$.
\end{enumerate}
\end{proposition}
Thus if we treat $z_{\mathbb F_q}(x)= \displaystyle\prod_{i=1}^{\infty} z(x^i)$, 
this proposition beautifully reflects the arithmetic nature of the field. 
To compare these numbers we make a table for small rank. The last row of this table is there in the work of Green.
\vskip3mm
\begin{center}
\begin{tabular}{|l|l|l|l|l|l|l|l|l|l|l|}
\hline
$z_k(n)$         &$z(1)$&$z(2)$&$z(3)$&$z(4)$&$z(5)$&$z(6)$&$z(7)$&$z(8)$&$z(9)$&$z(10)$ \\ \hline
$\mathbb{C}$   &1&3&6&14&27&58&111&223&424&817 \\ \hline
$\mathbb{R} $  &1&4&7&20&36&87&162&355&666&1367 \\ \hline 
$\mathbb F_{q}, q>n$        &1&4&8&22&42&103&199&441&859&1784 \\ 
\hline
\end{tabular}
\end{center}
\vskip2mm

\subsection{$z$-classes in hyperbolic unitary group}
In geometry the unitary groups used are over $\mathbb C$. Let $V$ be a vector space over $\mathbb C$ of dimension $n+1$. Hermitian forms are classified by signature (as in the case of quadratic forms over $\mathbb R$) and the corresponding groups are denoted as $U(r,s)$, where $r+s=n+1$. The group given by the identity matrix is compact unitary group denoted as $U_{n+1}(I_{n+1})=U(n+1,0)$. The genus number (which is the number of $z$-classes) of the compact special unitary group has been computed in~\cite{Bo} (see Theorem 3.1). We record the result here as follows:
\begin{proposition}
The number of $z$-classes in $U_{n+1}(I_{n+1})$ is $p(n+1)$.
\end{proposition}
The centralizers and $z$-classes in $U(n,1)$ have been described by Cao and Gongopadhyay in~\cite{CG} (see Section 4). However, they have not done explicit counting. Keeping in mind the spirit of this section we enumerate the number of $z$-classes in this group as well. We briefly recall the notation here and urge the interested reader to see the source and references therein for details. 
Let $V$ be an $(n+1)$-dimensional vector space over $\mathbb C$ with the hermitian form given by $\beta=\begin{pmatrix} -1 & 0 \\ 0 & I_{n} \end{pmatrix}$ which is of signature $(n,1)$. The unitary group is $U(n,1)=\{g \in GL(n+1,\mathbb{C}) \mid \tr{\bar{g}}\beta g=\beta\}$. 
Let $V_{-}$ be the vectors of negative length in $V$. The image of $V_{-}$ in the projective space $\mathbb P(V)$ is denoted as $\mathfrak H_{\mathbb C}^n$ which is a complex hyperbolic space. An element of the group is called elliptic if it has a fixed point on $\mathfrak H_{\mathbb C}^n$ and is said to be parabolic, respectively, hyperbolic, if it has exactly one, respectively, exactly two fixed points on the boundary of $\mathfrak H_{\mathbb C}^n$. Every element falls in one of these three classes. 
Using conjugation classification we know that if an element $g \in U(n,1)$ is elliptic or hyperbolic, 
then they are always semisimple. But a parabolic element need not be semisimple. 
However it has a Jordan decomposition $g=g_{s}g_{u}$, where $g_{s}$ is elliptic, hence semisimple, 
and $g_{u}$ is unipotent. In particular, if a parabolic isometry is unipotent, then it has minimal 
polynomial $(x-1)^2$ or $(x-1)^3$ and is called vertical translation or non-vertical translation respectively. The centralizers of these elements are described in~\cite{CG} Corollary 1.2 and we present the counting for $z$-classes here.
\begin{proposition}
In the group $U(n,1)$, 
\begin{enumerate}
\item the number of $z$-classes of elliptic elements is $\displaystyle \sum_{m=1}^{n+1} p(n+1-m)$.
 \item The number of $z$-classes of hyperbolic elements is $p(n-1)$.
 \item The number of $z$-classes of parabolic elements is $2+p(n-1)+p(n-2)$, in this case $n\geq 2$.
\end{enumerate}
\end{proposition}
\begin{proof}
Let $T\in U(n,1)$ be an elliptic element (see~\cite{CG} Section 4.1 (1)). Then $T$ has a negative class of eigenvalue 
say $[\lambda]$. Let $m=\dim (V_{\lambda})$ which is $\geq 1$. It follows from the conjugacy 
classification that all the eigenvalues have norm $1$, and there is a negative eigenvalue. 
All other eigenvalues are of positive type. 
Then $V=V_{\lambda}\perp V_{\lambda}^{\perp}=V_{\lambda}\perp (\perp_{i=1}^sV_{\lambda_i})$. 
Suppose $\dim(V_{\lambda_i})=r_{i}$, then 
$\mathcal Z_{U(n,1)}(T)=\mathcal Z_{U(V_{\lambda})}(T|_{V_{\lambda}})\times \prod_{i=1}^s U(r_{i})$. 
Now since $T|_{V_{\lambda}}$ is of negative type, so $\mathcal Z(T|_{V_{\lambda}})=U(m-1,1)$. 
Here $n+1=m+\sum_{i=1}^s r_i$, therefore $\sum_{i=1}^s r_i=n+1-m$. 
This gives that the number of $z$-classes of elliptic elements is $\sum_{m=1}^{n+1} p(n+1-m)$.

Now, suppose $T\in U(n,1)$ is hyperbolic (see~\cite{CG} Section 4.1 (2)). Then $V$ has an orthogonal decomposition  
$V=V_r \perp (\perp_{i=1}^k V_i)$, where $\dim(V_i)=r_i$ and $V_i$ is the eigenspace of $T$ 
corresponding to the similarity class of positive eigenvalue $[\lambda_i]$ with $|\lambda_i|=1$. 
The subspace $V_r$ is the two dimensional $T$-invariant subspace  spanned by the corresponding similarity class of null-eigenvalues $[re^{i\theta}], [r^{-1}e^{i\theta}]$ for $r>1$, respectively. 
Then $\mathcal Z_{U(n,1)}(T)=\mathcal Z(T|_{V_r}) \times \prod_{j=1}^k U(r_j)=
S^1 \times \mathbb{R} \times \prod_{j=1}^k U(r_j)$. Here $n+1=2+\sum_{j=1}^k r_j$, 
i.e., $\sum_{j=1}^k r_j=n-1$. Thus, the number of $z$-classes of hyperbolic elements is $p(n-1)$.
 
Let $T\in U(n,1)$ be parabolic (see~\cite{CG} Section 4.2). First, let $T$ be unipotent. If the minimal polynomial of $T$ is $(x-1)^2$, 
(i.e., $T$ is a vertical translation) then $\mathcal Z_{U(n,1)}(T)=U(n-1)\ltimes (\mathbb{C}^{n-1}\times \mathbb{R})$. 
If the minimal polynomial of $T$ is $(x-1)^3$, (i.e., $T$ is non-vertical translation) then 
$\mathcal Z_{U(n,1)}(T)=(S^1\times U(n-2))\ltimes ((\mathbb{R}\times \mathbb{C}^{n-2})\ltimes \mathbb{R})$. 
Hence there are exactly two $z$-classes of unipotents, one corresponds to the vertical translation and 
the other to the non-vertical translation.
Now assume that $T$ is not unipotent. Suppose that the similarity class of null-eigenvalue is $[\lambda]$. 
Then $V$ has a $T$-invariant orthogonal decomposition $V=V_{\lambda}\perp V_{\lambda}^{\perp}$, 
where $V_{\lambda}$ is a $T$-indecomposable subspace of $\dim(V_{\lambda})=m$, which is 
either $2$ or $3$. Then $\mathcal Z_{U(n,1)}(T)=\mathcal Z(T|_{V_{\lambda}})\times \mathcal Z(T|_{V_{\lambda}^{\perp}})$. 
For each choice of $\lambda$, there are exactly one choice for the $z$-classes of $T|_{V_{\lambda}}$ in $U(m-1, 1)$, 
i.e., $U(1,1)$ or $U(2,1)$. Note that $T|_{V_{\lambda}^{\perp}}$ can be embedded into $U(n+1-m)$. 
Hence it suffices to find out the number of $z$-classes of $T|_{V_{\lambda}}$ in $U(m-1,1)$. 
Hence the total number of $z$-classes of non-unipotent parabolic is $p(n-1)+p(n-2)$. 
Therefore the total number of $z$-classes of parabolic transformations is $2+p(n-1)+p(n-2)$ ($n\geq 2$).
\end{proof}

\subsection{$z$-classes in finite unitary group}
Now let us look at the finite unitary group in characteristic $\neq 2$. 
\begin{proposition}
\begin{enumerate}
\item The number of $z$-classes of unipotent elements in $U_n(q)$ is $p(n)$, which is equal to the number of $z$-classes of unipotent elements in $GL_n(q)$.
\item The number of $z$-classes of semisimple elements in $U_n(q)$ is equal to  the number of $z$-classes of semisimple elements in $GL_n(q)$ if $q>n$.
\end{enumerate}
\end{proposition}
\begin{proof}
Let $u=[J_1^{a_1} J_2^{a_2}\ldots J_n^{a_n}]$ be a unipotent element in $GL_n(q^2)$ written in Jordan block form. 
Wall proved the following membership test (see Case(A) on page 34 of~\cite{Wa}). Let $A \in GL_n(q^2)$ 
then $A$ is conjugate to ${}^t\bar{A}^{-1}$ in $GL_n(q^2)$ if and only if $A$ is conjugate to an element of $U_n(q)$. 
Since unipotents are conjugate to their own inverse in $GL_n(q^2)$, this implies $u$ is conjugate 
to ${}^t\bar{u}^{-1}$ in $GL_n(q^2)$. Hence $u$ is conjugate to an element of $U_n(q)$. 
Wall also proved that two elements of $U_n(q)$ are conjugate in $U_n(q)$ if and only if they are conjugate 
in $GL_n(q^2)$ (see also 6.1~\cite{Ma}). Thus, up to conjugacy, there is a one-one correspondence of unipotent 
elements between $GL_n(q^2)$ and $U_n(q)$. This gives that the number of unipotent conjugacy classes 
in $U_n(q)$ is $p(n)$, and it is same as that of $GL_n(q)$. Now, we note that 
$\mathcal Z_{U_n(q)}(u)=Q\prod_{i=1}^nU_{a_i}(q)$, where $|Q|=q^{\sum_{i=2}^n(i-1)a_i^2 + 2\sum_{i<j}ia_ia_j}$ (see~\cite{BG} Lemma 3.3.8). 
Clearly, the centralizers are distinct and hence can not be conjugate. Thus the number of unipotent $z$-classes in $U_n(q)$ is $p(n)$.

For semisimple elements we use Theorem~\ref{zsemisimple}. Since over a finite field, there is a unique hermitian form (see~\cite{Gv} Corollary 10.4) the third condition is irrelevant in counting. Over a finite field for $q>n$, we get that semisimple $z$-classes are characterized by simply $n=\sum_{i=1}^{k_{1}} d_im_i + \sum_{j=1}^{k_2} e_jl_j$, where $d_i$ is odd (being a degree of monic, irreducible, self-U-reciprocal polynomial, see Proposition~\ref{irr-self-U-rec}) and $e_j=2r_j$ is even. This corresponds to the number of ways $n$ can be written as $n=\sum_i a_ib_i$ when $q>n$ which is same as the number of semisimple $z$-classes in $GL_n(q)$ (see~\cite{Ku} Theorem 7.2). Notice that in the case $q\leq n$ we do not get enough field extensions required in the second condition of the Theorem~\ref{zsemisimple} and hence we may not get all partitions of $n$.
\end{proof}

\subsection{Proof of the Theorem~\ref{maintheorem2}}\label{prooftheorem2}     
We begin with recalling the parametrization of $z$-classes in $GL_n(q)$ described by Green in~\cite{Gr} (see the last para and equation (1) on page 407). It is given by $z_q(n)=t(n)$ which is the number of functions $\rho(\nu)$ satisfying $\sum_{\nu} |\rho(\nu)||\nu| =n $ and $\rho(\nu)$ is the partition-valued function on the non-zero partitions $\nu$ (here $\rho(\nu)$ is allowed to take value $0$). Following the proof there we note that $z_q(n)$ is 
$$\text{ the\ number\ of\ $z$-classes\ in\ $GL_n(q)$} = 
\displaystyle \sum_{[s]_z} \text{\ number of unipotent $z$-classes in}\, \mathcal Z_{GL_n(q)}(s),$$
and $\mathcal Z_{GL_n(q)}(s)$ is a product of $GL_m(q')$ where $m\leq n$ and $\mathbb F_{q'}$ is a field extension of $\mathbb F_q$ of degree $\leq n$. Green also clarifies that this formula works for ``sufficiently large $q$''. We use the same strategy and show that the $z$-classes in $U_n(q)$ are parametrised by the same function for $q>n$.
Recall that if $g=g_sg_u$ is a Jordan decomposition of $g$ 
then $\mathcal Z_{U_n(q)}(g)=\mathcal Z_{U_n(q)}(g_s)\cap \mathcal Z_{U_n(q)}(g_u)
=\mathcal Z_{\mathcal Z_{U_n(q)}(g_s)}(g_u)$, and the structure of $\mathcal Z_{U_n(q)}(g_s)$ 
in Theorem~\ref{zsemisimple} implies that 
$$\text{ the\ number\ of\ $z$-classes\ in\ $U_n(q)$} = 
\displaystyle \sum_{[s]_z} \text{\ number of unipotent $z$-classes in}\, \mathcal Z_{U_n(q)}(s),$$
where the sum runs over semisimple $z$-classes and is same as for $GL_n(q)$. Now, we know that the number of unipotent $z$ classes in $U_m(q')$ and $GL_m(q')$ is same and thus unitary groups appearing in the centralizer of a semisimple element has the same effect in counting as for the general linear group. Hence the number of $z$-classes in $U_n(q)$ is the same as the number of $z$-classes in $GL_n(q)$.

 
\section{Examples}\label{example}
We give some examples which highlight the need of the property FE on the field. 

\begin{example} 
Over field $\mathbb Q$, there are infinitely many non-conjugate maximal tori in $GL_n$. Since a maximal torus is centralizer of a regular semisimple element in it, we get an example of infinitely many $z$-classes. For the sake of clarity let us write down this concretely when $n=2$.  

The group $GL_2(\mathbb Q)$ has infinitely many semisimple $z$-classes. 
For this, we take $f(x)\in \mathbb Q[x]$ a degree $2$ irreducible polynomial, then the centralizer of the companion matrix $C_f\in GL_2(\mathbb Q)$ is isomorphic to $\mathbb Q_f^*$, where $\mathbb Q_f=\mathbb Q[x]/<f(x)>$, a degree $2$ field extension of $\mathbb Q$. For example, if we take $f(x)=x^2-p$ where $p$ is a prime we get infinitely many non-isomorphic degree $2$ field extensions of $\mathbb Q$. Since these are not isomorphic the corresponding centralizers $\mathbb Q_f^*$ can not be conjugate. This way we get infinitely many $z$-classes of semisimple elements, in fact, this gives infinitely many non-conjugate maximal tori in $GL_2(\mathbb Q)$. 

Consider $\mathbb F= \mathbb Q[\sqrt d]$, 
a quadratic extension. We embed $GL_2(\mathbb Q)$ in $U_4$ with respect to the hermitian 
form $\begin{pmatrix} & I_2 \\ I_2 & \end{pmatrix}$ given 
by $$A\mapsto \begin{pmatrix} A & \\ & \tra {\bar A}^{-1}\end{pmatrix}.$$
This embedding describes maximal tori in $U_4$ starting from that of $GL_2$. Yet again, non-isomorphic degree $2$ field extensions would give rise to distinct $z$-classes. In turn, this gives infinitely many $z$-classes (of semisimple elements) in $U_4$.
\end{example}
\begin{example}
For $a\in \mathbb F^*$, consider a unipotent element $u_a=\begin{pmatrix} 1&a\\ &1\end{pmatrix}$ in $SL_2(\mathbb F)$. 
Then $\mathcal Z_{SL_2(\mathbb F)}(u_a)= \left\{ \begin{pmatrix} x&y\\ &x\end{pmatrix} \mid x^2=1,  y\in \mathbb F \right\}$. 
Then, $u_a$ is conjugate to $u_b$ in $SL_2(\mathbb F)$ if and only if $a\equiv b \imod {(\mathbb F^*)^2}$. Let $\mathbb F$ be a (perfect or non-perfect) field with $\mathbb F^*/(\mathbb F^*)^2$ infinite. Then this would give an example, where we have infinitely many conjugacy classes of unipotents but still, they are in a single $z$-class. 
\end{example}
\begin{example}\label{example3}
Over a finite field $\mathbb F_q$, if $q$ is not large enough we may not have as many finite extensions available as required in part 2 of Theorem~\ref{zsemisimple}. Thus we expect a smaller number of $z$-classes. We use GAP~\cite{GAP4} to calculate the number of $z$-classes for small order and present our findings below.  
\vskip3mm
\begin{center}
\begin{tabular}{|c|l|l|l|l|l|l|}
\hline
$z_{\mathbb F_q}(2) $    & $q=2$     & $q=3$&$q=4 $&$q=5$&$q=7$&$q=9$ \\ \hline
$GL_2(q)$ & $3$ & $4$ & $4$ & $4$ & $4$ & $4$ \\ \hline
$U_2(q)$  & $3$ &$4$ & $4$ & $ 4$& $4$ & $4$ \\ \hline 
\end{tabular}
\end{center}

\vskip3mm
\begin{center}
\begin{tabular}{|c|l|l|l|l|l|l|}
\hline
$z_{\mathbb F_q}(3) $      &$q=2$   &$q=3$& $q=4$ & $q=5$&$q=7$&$q=9$ \\ \hline
$GL_3(q)$ & $5$& $7$ & $8$ & $8$ &$8$ &$8$ \\ \hline
$U_3(q)$  & $7$ & $8$ & $8$ & $8$ &$8$ &$8$ \\ \hline 
\end{tabular}
\end{center}

\vskip3mm
\begin{center}
\begin{tabular}{|c|l|l|l|l|l|}
\hline
$z_{\mathbb F_q}(4) $     &$q=2$    &$q=3$& $q=4$ & $q=5$&$q=7$ \\ \hline
$GL_4(q)$   & $11$& $19$ & $21$ & $22$& $22$ \\ \hline
$U_4(q)$  & $15$ & $22$ & $22$ & $22$ &$22$ \\ \hline 
\end{tabular}
\end{center}
\vskip2mm
\noindent Thus we demonstrate the following:
\begin{enumerate}
\item When $q\leq n$ the number of $z$-classes in $GL_n(q)$ and $U_n(q)$ are not given by the formula in Theorem~\ref{maintheorem2}.
\item When $q\leq n$ the number of $z$-classes in $GL_n(q)$ and $U_n(q)$ need not be equal.
\end{enumerate}
\end{example}


\end{document}